\documentclass{article}
\usepackage[utf8]{inputenc}
\usepackage{amsmath}
\usepackage[T1]{fontenc}    
\usepackage{hyperref}       
\usepackage{url}            
\usepackage{booktabs}       
\usepackage{amsfonts}       
\usepackage{nicefrac}       
\usepackage{microtype}      
\usepackage{lipsum}
\usepackage{graphicx}
\usepackage[english]{babel}
\usepackage{amsthm}
\usepackage{tabularx}
\usepackage{array}

\graphicspath{ {./images/} }

\usepackage{mathtools}

\usepackage{amssymb}

\newtheorem{theorem}{Theorem}[section]

\newtheorem{conjecture}[theorem]{Conjecture}

\usepackage{listings}
\usepackage{xcolor}
\title{Sylvester's Conjecture and  the Egyptian Fractions}
\author{Keneth Adrian P. Dagal \\ 
  \texttt{kendee2012@gmail.com} \\}

\begin{document}

\maketitle

\begin{abstract}
This paper attempts to prove the Sylvester's conjecture using Egyptian Fractions with two key ingredients. First, creating a set of operators that completely generates all possible Egyptian fraction of 1. And second, to detect patterns in every operator that surely will generate a new number which are relatively prime to all that came before. 
\end{abstract}

\section{Introduction}

Gimbel and Jamora \cite{Sylvester} provide extensive information about the modern era of research on odd perfect number (OPN) up to the year of its publication and focus on Sylvester's career and interest on the nonexistence of odd perfect number. Here, we update such information.

We say that a positive integer is perfect if $$\sigma(n) = 2n,$$ where $\sigma(n)$ is defined as the sum-of-divisor function. Euler proved that all even perfect numbers are of Euclid's form $2^{\alpha-1}(2^\alpha-1)$ where $2^\alpha-1$ is a Mersenne prime. As of June 2020, the largest known Mersenne prime is $2^{82,589,933}-1$, which is the 51st Mersenne Prime \cite{Mersenne}. For the odd case, Ochem and Rao \cite{Ochem} proved that if an OPN $n$ exists, $n > 10^{1500}$. In this direction, we still hope to find the first odd perfect number. And for now, we conclude that we only know 51 perfect numbers which are all even.

However, Sylvester believes that there is no odd perfect number. So we state:

\begin{conjecture}( Sylvester's Conjecture)\\
There is no odd perfect number.
\end{conjecture}

We generalize the conjecture to:

\begin{conjecture}
There is no odd m-perfect number for $m \geq 2$.
\end{conjecture}

If $m=2$, then we can simply say perfect instead of $2$-perfect.A positive integer $n$ is called Multiply perfect number ( multiperfect number or $m$-perfect) if $\sigma(n)= mn$. The generalized conjecture is possible since all known multiply perfect numbers are even.\cite{Multi}  Also, we extend the definition of sum-of-divisor function to

$$\sigma_{s}(n) = \sum_{d|n} d^s.$$

And consider the cases where $s = -1, 0,$and $1$. Note that $\sigma_{1}(n)=\sigma(n)$, $\sigma_{0}(n)=\tau(n)$, and $\sigma_{-1}(n)=I(n)$.\cite{Knill}

\section{Preliminaries}

Euler proved that if $n$ is an OPN, it must be of the form $$n = p^\alpha\prod_{i=1}^{\omega(n)-1}{q_i^{2\beta_i}}$$ 
where $p$ and $q_i$'s are distinct odd primes, $p\equiv \alpha \equiv 1(\mod4)$ and $\omega(n)$ means the number of distinct prime factors of $n$.

As of June 2020, the bounds for $n$ are $$ 10^{1500} < n < \frac{2^{4^{\omega(n)}}}{10^{12}P^2}$$ where $P$ is the largest prime divisor.
From the upper bound, Neilsen \cite{Neilsen} derived the inequality $\omega(n) \geq 10$. Also, Ochem and Rao \cite{Ochem} obtain the inequality $\Omega(n) \geq 101$ and its largest component,either $q^{2\beta_i}$ or the special prime  $p^{\alpha}$, is greater than $10^{62}$.

There is a number of mathematicians in the last two centuries who derived results from parts of Euler's form of an odd perfect number, one example 
is that of Sylvester. He proved that an OPN $n$ is not divisible by 105. The author refers the reader to Gimbel \cite{Sylvester}, Knill \cite{Knill}, and Neilsen \cite{Neilsen} for information about various results both conditional and unconditional restrictions of an odd perfect number.

Sylvester's Conjecture can be expressed as: If $\sigma(n) = 2n$, then $n$ is even. This is logical equivalent to the conjecture below: 

\begin{conjecture}
If $n$ is odd, then $\sigma(n)  \neq 2n$.
\end{conjecture}

We prove the conjecture by assuming $n$ is odd, and $\sigma(n)= 2n$ and derive a contradiction. From this, we consider

$$\sigma_{-1}(n)= \sum_{d|n} \frac{1}{d} = 2$$

For $d \geq 3$, it is sufficient to consider the equation

$$1=\sum_{d|n} \frac{1}{d}$$

If the above equation exists, then so an odd perfect number. Our objective is to construct the equation above from a seed equation.

 $$1= \frac{1}{3}+\frac{1}{5}+\frac{1}{7}+\frac{1}{9}+\frac{1}{11}+\frac{1}{15}+\frac{1}{35}+\frac{1}{45}+\frac{1}{231}\cite{Shiu}$$ 

The equation above has 9 parts, denoted by $k$. And suppose all denominators, denoted by $u_i$, are all the proper divisors of an arbitrary $n$, then we can see that $\sigma_{0}(n)-1 = k$. Notice that the $lcm(u_1, u_2, ... ,u_9)= 3^2\cdot5\cdot7\cdot11 = 3465=n$. Since $\sigma_{0}(3465)= 24$ and $k=9$, then 3465 is an abundant number with 14 missing divisors. With this, we focus on the following:
\begin{itemize}
    \item  A set of operators $\gamma$ in generating a complete set of all Egyptian fractions equals 1.
    \item  Suppose there are $x$ number of operators that satisfies the above condition, there always exists a path at which one will be able to generate a $u_i'\in E'$ such that $gcd(u_i,u_i') =1$ for all $u_i \in E$ with $\gamma_i : E \rightarrow E'$, where $E$ and $E'$ are Egyptian fractions.
\end{itemize}

Dagal \cite{DagalOp} provides a definition of Egyptian fraction and examples of Egyptian Fraction Operators. And if the two above points are met, then we have an affirmative answer to Sylvester's Conjecture.

\section{Egyptian Fractions}

Let us re-state Shiu's theorem \cite{Shiu} where the seed equation came from.
\begin{theorem}
 The solution to the  Diophantine Equation $$ 1=\sum_{i=1}^l \frac{1}{u_i},$$ where $3 \leq u_1 < u_2 < ... < u_l$ and  $gcd( u_i, 2)=1$ must have $l \geq 9$ and for $l=9$, the only solutions are the following:
 
 $$1= \frac{1}{3}+\frac{1}{5}+\frac{1}{7}+\frac{1}{9}+\frac{1}{11}+\frac{1}{15}+\frac{1}{35}+\frac{1}{45}+\frac{1}{231} $$
 $$1= \frac{1}{3}+\frac{1}{5}+\frac{1}{7}+\frac{1}{9}+\frac{1}{11}+\frac{1}{15}+\frac{1}{33}+\frac{1}{45}+\frac{1}{385} $$
 $$1= \frac{1}{3}+\frac{1}{5}+\frac{1}{7}+\frac{1}{9}+\frac{1}{11}+\frac{1}{15}+\frac{1}{21}+\frac{1}{231}+\frac{1}{315} $$
 $$1= \frac{1}{3}+\frac{1}{5}+\frac{1}{7}+\frac{1}{9}+\frac{1}{11}+\frac{1}{15}+\frac{1}{21}+\frac{1}{165}+\frac{1}{693} $$
 $$1= \frac{1}{3}+\frac{1}{5}+\frac{1}{7}+\frac{1}{9}+\frac{1}{11}+\frac{1}{15}+\frac{1}{21}+\frac{1}{135}+\frac{1}{10395} $$
\end{theorem}

Remember that an OPN $n$ is not divisible by $105= 3\cdot 5\cdot 7$. Thus, we must perform the operations $\gamma$ considering the known results and restrictions for an OPN. Perhaps, with this consideration, a contradiction might be observed immediately. Take note that $k$ must be odd. 

Consider the table below for 1 (without any restrictions for Egyptian fractions), that is, $u's$ can be repeated and can be odd or even.

\begin{center}
    \begin{tabular}{|c|c|c|c|c|c|c|}
    \hline
       $k=1$ & $2$ & $3$ & 4& $\cdots$ \\ \hline
        1 & $(2,2)$ & $(2,3,6)$  &$\cdots$& $\cdots$ \\ \hline
         &  & $(2,4,4)$ & &$\cdots$  \\ \hline
         &  & $(3,3,3)$ &  & $\cdots$ \\ \hline
    \end{tabular}
\end{center}

In the table above, for ease of notation, we use the identity $$ (u_1, u_2, \cdots, u_k) = \sum_{i=1}^k\frac{1}{u_i}.$$ We define $S(k)$ be the set that contains all Egyptian fractions of 1 with $k$ parts and we denote the number of elements in $S(k)$ be $|S(k)|$.
Let $S$ be the set that contains all Egyptian fractions of 1. Then we have $$S= \bigcup_{i=1}^\infty S(k).$$

Our path is to start with a table with the least number of restrictions accompanied with a set of operators that generates all elements in the table. Whenever we add two restrictions, say we add the property of being odd and distinct, then $k\geq 9$. and we drop all elements from $k=1$ to $k=8$ by Shiu's Theorem and all even $k$. In addition, there will be a modified set of operators that will preserve the odd property of the $u_i$'s.

\section{Egyptian Fractions of 1 with k parts}

We start by finding a bound for $|S(k)|$. It is trivial that $|S(1)|=|S(2)|=1$ and $|S(3)|=3$. Since $S(k) = \{ (u_1, u_2 , ... , u_k) \,|\, u_1 \leq u_2 \leq \cdots \leq u_k\}$, the trivial bounds for $u_1$ and $u_k$ are $$ u < u+1 \leq u_1 \leq k\cdot u \leq u_k.$$

The bounds came from the operator $\gamma: (u) \rightarrow (u+1, u(u+1))$. This operator makes one (1) part to two (2) parts. The upper bound for $u_k$ depends on the set of operators and the number of times we operate, so we state the upper bound based on the case.

Suppose we consider $\gamma: (u) \rightarrow (u+1, u(u+1))$. $\gamma^1: (u) \rightarrow (u_1, u_2)$ and $\gamma^2: (u) \rightarrow (u_1', u_2', u_3')$.  The bounds for $u_1'$ and $u_3'$ with $\gamma^2$ and with respect to $u$ are $$ u+1 \leq u_1' \leq k\cdot u \leq u_3' \leq u^2(u+1)^2+u(u+1).$$

We generalize the operator $\gamma$ to $$\gamma_d^{k-1}: (u) \rightarrow (u_1', u_2', \cdots, u_k')$$ where  $u \in \mathbb{Z^+}$, $\gamma_d: (u) \rightarrow (u+d, \frac{u}{d}(u+d))$, and $d | u$. For this operator, notice that if $k=2$, $|S(2)| \geq \sigma_0(u)$. Though our concern before was when $u=1$,now we investigate when $u$ is any positive integer instead. Also, though the operator $\gamma_d^{k-1}$ acts like a recurrence equation with some combinatorial activities in each iteration, it is understood that $d_i$ denotes all possible divisors for each $u_{(i)}$ where $i$ is the $i$th step of the $\gamma_d^{k-1}$ operator.

\begin{theorem}
Let $\gamma_d^{k-1}: (u) \rightarrow (u_1', u_2', \cdots, u_k')$. Then $$|S(k)| \geq s(k) = \sum_{j=1}^{s(k-1)} \sigma_0((u_{1(i-1)}+d_{(i-1)})_j)$$ where $s(2)= \sigma_0(u)$.
\end{theorem}

\begin{proof}
$\sigma_0(u_{1(i)})$ counts the number of divisors $d_{(i)}$ which in turn counts some elements in $S(i+1)$.At any step, we only operate $u_{1(i)}$ with $\gamma_d: (u_{1(i)} \rightarrow ((u_{1(i)}+d_{(i)}), \frac{u_{1(i)}}{d_{(i)}}(u_{1(i)}+d_{(i)})))$. This guaranteed uniquely generated Egyptian fractions at each stage until $k$. Doing this process repeatedly until $k$, we get the result above.
\end{proof}

The inequality is due to the fact that we are sure that there are some integer $u$ where $|S(2)| > s(2)$ and the operator $\gamma_d^{k-1}$ does not generate all of that in $S(k)$. An example would be when $u=6$, a perfect number. With $\gamma$, we have $\sigma_0(6)=4$ and $k=2$, but $$S(2) = \{ (7,42),(8,24), (9,18), (10, 15), (12,12)\}.$$ The operator $\gamma$ does not generate the case $(10, 15)$.This was observed and checked due to the bounds for $u_1$ which are $ 7 \leq u_1 \leq 2(6)=12$. 

This suggests to find another operator $\mathbb{O} \neq \gamma$ which covers these cases where $\gamma$ cannot reach. The first goal is the have a set of operators that completely generates all Egyptian fractions of a given $u$. In this generalized form, we are interested in the case when $u=1$ and when an operator $O: (u_1, u_2, \cdots, u_a) \rightarrow (u_1', u_2', \cdots, u_k')$ where all $u_i$ and $u_i's$ are odd numbers such that $a \leq k$. One such example was given by Dagal:

\begin{theorem}
Let $r = q+d$, and $s = qr-d$.  The operator $O: (s, rs)\rightarrow (qr, qs)$ is an odd parity preserving equation if and only if the integer $q > 1$ is odd, and the value of $d$ is a positive even number.
\end{theorem}

\section{Concluding Remarks}

As of its present form, we have the operator $\gamma_d^{k-1}$ to generate some Egyptian fraction of a given $(u)$. We are looking of ways to have a set of operators and prove it complete. And then, we modify this operators based the restrictions we want to include. And with this set of operators, we investigate more. All this note has offered the reader are the following: 

\begin{itemize}
    \item An approach to settle the odd perfect number conjecture.
    \item The generalized $\gamma$ operator to generate, if not all, most Egyptian fractions from $k=1$ to any arbitrary positive integer $k$.
\end{itemize}

\section{Acknowledgement}

The author appreciates Jose Arnaldo Dris for the valuable conversation about odd perfect numbers, Sam Ishmaiel Haider for having faith in the author for this endeavor,  and Selvavinayagam Velusamy and  Mohammad Abdul Aziz Qureshi for moral support.Lastly, the author would like to thank the Almighty God for everything therein.

\end{document}